\documentclass[a4paper,10pt]{article}

\usepackage{amscd}
\newcommand{\ds}{\displaystyle}
\newcommand{\lra}{\rightarrow}

\renewcommand{\P}{\frak{P}}

\usepackage{color}
\usepackage{amssymb}
\usepackage{amsthm}
\usepackage{amsmath}
\usepackage{latexsym}

\newtheorem{Thm}{Theorem}[section]
\newtheorem{Prop}[Thm]{Proposition}
\newtheorem{Cor}[Thm]{Corollary}

\newtheorem{Lem}[Thm]{Lemma}
\newtheorem{Rem}[Thm]{Remark}

\newcommand{\bQ}{\overline{\mathbb{Q}}}
\newcommand{\Q}{\mathbb{Q}}
\newcommand{\Z}{\mathbb{Z}}
\newcommand{\F}{\mathbb{F}}

\newcommand{\M}{{\cal M}}
\renewcommand{\l}{{\frak l}}

\renewcommand{\O}{{\cal O}}

\renewcommand{\ker}[1]{{\rm ker}\hspace{0.1cm}{#1}}

\renewcommand{\P}{{\frak P}}
\newcommand{\p}{{\frak p}}

\newcommand{\mapright}[1]{%
  \smash{\mathop{%
  \hbox to 1cm{\rightarrowfill}}\limits^{#1}}}
\newcommand{\mapleft}[1]{%
  \smash{\mathop{%
  \hbox to 1cm{\leftarrowfill}}\limits^{#1}}}

\newcounter{CounterEQUlabel}
\newcommand{\EQUlabel}[1]{\label{#1}
\ifcase \theCounterEQUlabel
   \relax
  \or
   \hspace{1em}
   \mbox{{\color{blue} \tiny$\langle$\rmfamily#1$\rangle$}}\index{zzz#1@#1}
  \fi
}

\title{On the class numbers of the fields of the 
$p^n$-torsion points of elliptic curves over $\Q$
\\
{\small {\it Dedicated to Professor Hirotada Naito's 60th birthday}}
}

\author{Fumio SAIRAIJI and Takuya YAMAUCHI
\footnote{The second author
is partially supported by JSPS Grant-in-Aid for Scientific Research (C) No.15K04787.}
}

\date{}

\begin{document}

\maketitle
\begin{abstract}
Let $E$ be an elliptic curve over $\Q$ 
which has multiplicative reduction at a fixed prime $p$. 
Assume $E$ has multiplicative reduction or potentially good reduction at 
any prime not equal to $p$. 
For each positive integer $n$ we put $K_n:=\Q(E[p^n])$. 
The aim of this paper is to extend the authors' previous results in \cite{SY} 
concerning with the order of the $p$-Sylow group of the ideal class group of $K_n$ to more general setting. 
We also modify the previous lower bound of the order given in terms of the Mordell-Weil rank of $E(\Q)$ and 
the ramification related to $E$. 
\end{abstract}
\section{Introduction}
This article is a sequel of \cite{SY}. 
Let $p$ be a prime number and $E$ be an elliptic curve over $\Q$.  
For each positive integer $n$, 
we consider the field $K_n$ generated by the coordinates of points 
on $E[p^n]$ over $\Q$. 
In \cite{SY} the authors studied a lower bound of 
the $p$-part of the class number $h_{K_n}$ of $K_n$ in terms of the Mordell-Weil rank of $E(\Q)$ when $E$ has prime conductor $p$.
The present article extends this result to a more extensive class of elliptic curves over $\Q$. 

For such an elliptic curve, we will carry out a similar estimation done in \cite{SY} but at the same time 
we give an improvement of the method of the estimation. 
%This enable us to gain a lower bound which is the half size of the previous estimation in most cases. 
As we have done in \cite{SY} the lower bound will be given in terms of the Mordell-Weil rank and the information coming from the ramification related to $E$. Our formula is reminiscent of Iwasawa's class number formula for $\Z_p$-extension. 
In fact we have an explicit class number formula in a special case (see Corollary \ref{cor}).    

Our study is motivated by the works of Greenberg \cite{G} and Komatsu-Fukuda-Yamagata \cite{FKY} who have studied a lower bound of Iwasawa invariants for CM fields in terms of the Mordell-Weil group 
of the corresponding CM abelian varieties. 
We have pursued an analogue for non-CM elliptic curves since \cite{SY}. 
 
To state our main theorem we introduce our notation. 
The Mordell-Weil theorem asserts that $E(\Q)$ is a finitely generated 
abelian group. 
Thus there exists a free abelian subgroup $A$ of $E(\Q)$ of finite rank such that $A+E(\Q)_{{\rm tors}}=E(\Q)$. We denote the rank of $A$ by $r$. 
We put $G_n:={\rm Gal}(K_n/\Q)$ and 
$L_n:=K_n([p^n]_E^{-1}A)$, 
where $[p^n]_E$ is the multiplication-by-$p^n$ map on $E$. 
We denote generators of $A$ by $P_1,\ldots,P_{r}$. 
For each $j$ in $\{1,\ldots,r\}$ we take a point $T_j$ of $E(L_n)$ satisfying 
$$
[p^n]_E(T_j)=P_j.
$$
Then we have $L_n=K_n(T_1,\ldots,T_{r})$. 
The Galois action on $\{T_j\}_j$ naturally induces an injective 
$G_n$-homomorphism 
$$
\Phi_n:~\mbox{Gal}(L_n/K_n) \rightarrow E[p^n]^{r}~:~
\sigma \mapsto ({}^\sigma T_j- T_j)_j
$$
(cf.\ \cite{L}, p.\ 116).
In particular, the degree $[L_n:K_n]$ is equal to a power of $p$. 

We denote the maximal unramified abelian extension of $K_n$ 
by $K_n^{{\rm ur}}$. 
We define the exponent $\kappa_n$ by 
$$
[L_n \cap K_n^{{\rm ur}}:K_n]=p^{\kappa_n}.
$$

Assume that $E$ has multiplicative reduction at $p$ 
and $E$ has multiplicative reduction or potentially good reduction at 
any prime $\ell\ne p$. 

Then the main theorem of this article is the following theorem. 
%
%
%
%Theorem 1.1
%
\begin{Thm}\label{main}
Assume that $G_n\simeq {\rm GL}_2(\Z/p^n\Z)$ for each $n\ge 1$ and $p{\not|}{\rm ord}_p(\Delta)$, where $\Delta$ is the minimal discriminant of $E$.  
The following inequalities hold: 
\begin{enumerate}
\item Assume that $p$ is odd. Then for any $n\ge 1$,  
$$\kappa_n \geq 2n(r-1)-2\sum_{\ell\ne p}\nu_\ell,$$
 where we put
$$
\nu_\ell:=
\left\{
\begin{array}{ll}
\min\{{\rm ord}_p({\rm ord}_\ell(\Delta)),n\}&
\mbox{if $E$ has split multiplicative reduction at $\ell \ne p$}\\
0&\mbox{otherwise}. 
\end{array}
\right.
$$
 
\item  Assume that $p=2$. Then for any $n\ge 1$, 
$$\kappa_n \geq 2n(r-1)-2(r_{2,n}-2)-\delta_2-2\sum_{\ell\ne p}\nu_\ell,$$
where 
$r_{2,n}=1,2$ according as $E(\Q)/E(\Q)\cap [2^n]_E(E(\Q_p))$ 
is cyclic or not, and we put
$$
\nu_\ell:=
\left\{
\begin{array}{ll}
\min\{{\rm ord}_p({\rm ord}_\ell(\Delta)),n\}&
\mbox{if $E$ has split multiplicative reduction at $\ell \ne 2$}\\
1&\mbox{if $E$ has potentially good reduction at $\ell \ne 2$ and $n=1$,}\\
&
\mbox{or if $E$ has non-split multiplicative reduction at $\ell \ne 2$}\\
&
\mbox{and $\mbox{ord}_\ell(\Delta)$ is even,} \\
0&\mbox{otherwise}. 
\end{array}
\right.
$$
and 
$$\delta_2:=\left\{
\begin{array}{cc}
2 & \mbox{if $n=1$ and $r_{2,1}=1$}\\ 
0 & \mbox{otherwise}
\end{array}\right.
$$ 
\end{enumerate}
\end{Thm} 

\begin{Cor}
Assume that the conductor of $E$ is equal to $p$. 
Then 
$$
\kappa_n=\left\{ \begin{array}{ll}2n(r-1)+2\nu&(n>\nu)\\
2nr&(n\leq \nu)
\end{array}
\right.
$$
for some integer $\nu\ge 0$ $($which depends only on $P_j)$. 
\label{cor}
\end{Cor}

We explain the conditions imposed on $E$. 
Put $$n_0:=
 \left\{
\begin{array}{ll}
 1 & {\rm if}\ p>3 \\
 2 & {\rm if}\ p=3 \\
 3 & {\rm if}\ p=2.
\end{array}\right. $$
It is known that  $G_n\simeq {\rm GL}_2(\Z/p^n\Z)$ for $n\le n_0$ implies 
$G_n\simeq {\rm GL}_2(\Z/p^n\Z)$ for all $n\ge 1$ (cf.\ \S 1 of \cite{DD}). 
Thus the assumption in theorem \ref{main} can bee replaced by 
$G_n\simeq {\rm GL}_2(\Z/p^n\Z)$ for $n \le n_0$. 
For a given prime number $p$, there is a criterion for $G_n$ ($n\le n_0$) 
to be isomorphic to ${\rm GL}_2(\Z/p^n\Z)$ (see \cite{DD} for $p=2$,  
\cite{Elkies} for $p=3$, and \cite{Ser} for $p\ge 5$). 
The condition  $p{\not|}{\rm ord}_p(\Delta)$  is automatically satisfied when $p>5$ and 
$E$ is a semistable  elliptic curve (cf.\ Th\'eor\`eme 1, p.\ 176 of \cite{MO}). 
%The meaning of $\delta_2$ will be revealed in \S 3. 
%
%
%
%
%
\par
\bigskip
While preparing this paper, Hiranouchi \cite{H} generalized Theorem \ref{main} (1) to the case 
where $p>2$, $E(\Q_p)[p]=\{0\}$, $G_1\simeq \mbox{GL}(\Z/p\Z)$. 
He uses the structure theorem of $E(\Q_p)$ and the formal group for 
$E$ which plays a substitution of Tate curves. 
He also shows $E(\Q_p)[p]=\{0\}$ for $p>2$ under the assumption of Theorem \ref{main}. 
\par
\bigskip
The organization of this paper is as follows. 
In \S 2, 
we study the extension $L_n/K_n$. 
To do this we modify Bashmakov's result \cite{bas} from Lang \cite{L} for our elliptic curves. 
Then we investigate the degree $[L_n:K_n]$ of the extension $L_n/K_n$. A key point is to show the equality $L_1\cap K_\infty=K_1$. 
To do this we separate the situation into the case when $p$ is odd and the case when $p=2$. 
The former case will be done in \S 2, but the latter case 
will be devoted to \S 6 because of the particular treatment in which being $p=2$ causes.  
In \S 3 and \S 4, we investigate the degree of the $p$-adic completion of $L_n$ over the one of $K_n$, 
which is used for the estimate of the inertia group in \S 5. 
We give the proof of Theorem \ref{main} in \S 5, 
and we give numerical examples of $\kappa_n$ in \S 7. 

\textbf{Acknowledgments.} We would like to express our deep appreciation to Professors Matsuno Kazuo and Toshiro Hiranouchi, and 
Yoshiyasu Ozeki for comments in our previous version.  We thank also the referee for helpful suggestions and collections which 
are useful for improving contents and presentations. 
\section{The extension $L_n$ over $K_n$}
\label{extension}

In this section, we extend some results in \cite{SY} which has been done by the arguments essentially based on Bashmakov \cite{bas} 
(cf.\ Lemma 1 of \cite[p.\ 117]{L}). 

Let us keep the notation in \S 1 and throughout this paper we assume our elliptic curve $E$ always satisfies the condition in 
Theorem \ref{main}. 
Put $K_\infty:=\ds\cup_{n\ge 1}K_n$, 
$L_\infty:=\ds\cup_{n\ge 1}L_n$ and $G_\infty:={\rm Gal}(K_\infty/\Q)$. 
For each $n\ge 1$ let us consider the $G_n$-homomorphism 
$$\Phi_n:{\rm Gal}(L_n/K_n)\lra E[p^n]^r.$$ 
It follows that the $G_\infty$-homomorphism 
$$\Phi_\infty:=\ds\varprojlim_n \Phi_n:{\rm Gal}(L_\infty/K_\infty)\lra T_p(E)^r$$ 
is injective and  the image is a closed subgroup.
We are concerned with the order of the image of $G_n$-homomorphism $\Phi_n$. 
As we will see below, $\Phi_1$ controls $\Phi_\infty$ and then the information for $\Phi_n$ comes up from $\Phi_\infty$. 

To obtain a lower bound of the class number in question, 
we need to study that the image of $\Phi_n$ to guarantee 
the degree $[L_n:K_n]$ is large enough. 
We will prove that $\Phi_n$ is an isomorphism for any prime $p$ and $n\ge 1$.  
\begin{Thm}\label{160120a}
Assume that $G_1 \simeq {\rm GL}_2(\Z/p\Z)$. 
Then, $\Phi_1$ is an isomorphism for any prime $p$. 
In particular, the equality $[L_1:K_1]= p^{2r}$ holds. 
\end{Thm}
\begin{proof}
The proof given here is almost identical with the proof of Theorem 2.4 in \cite{SY}. 
Therefore we only explain a key point.  
Since $G_1\simeq {\rm GL}_2(\Z/p\Z)$, the Galois cohomology $H^1(G_1,E[p])$ vanishes by \cite{LW}. 
Then there is an injective homomorphism 
$$A/[p]_EA\hookrightarrow {\rm Hom}_{G_1}({\rm Gal}(L_1/K_1),E[p])$$
(see l.\ 6 in p.\ 283 of \cite{SY}).
Therefore we have $\sharp {\rm Hom}_{G_1}({\rm Gal}(L_1/K_1),E[p])\ge p^r$, 
where $r$ is $\Z$-rank of $A$. 
On the other hand ${\rm Gal}(L_1/K_1)\simeq E[p]^s$ for some $s\le r$. Then we see that 
$${\rm Hom}_{G_1}({\rm Gal}(L_1/K_1),E[p])\simeq {\rm End}_{G_1}(E[p])^s\simeq (\Z/p\Z)^s$$
which implies $s\ge r$. Hence $s=r$ and it turns that $\Phi_1$ is 
an isomorphism. 
\end{proof}
Theorem \ref{160120a} is different from Theorem 2.4 of 
\cite{SY} at the point where we omit the assumption that $N$ is prime 
and $p>2$. 

To show that $\Phi_n$ is an isomorphism, we have the following lemma. 
\begin{Lem}\label{any-case}
Assume that $G_n\simeq {\rm GL}_2(\Z/p^n\Z)$ for $n \ge 1$. 
Then, 
the equality $L_1\cap K_\infty=K_1$ holds for any prime $p$.  
\end{Lem}
\begin{proof}
In case when $p$ is odd prime 
the assertion follows from Lemmas 2.1 and 2.2 of \cite{SY}. 
In case when $p=2$ it follows from Theorem \ref{p=2}. 
\end{proof}
\begin{Thm}\label{0813a}
Assume that $G_n\simeq {\rm GL}_2(\Z/p^n\Z)$ for $n \ge 1$. 
Then, $\Phi_n$ is an isomorphism for $n\ge 1$ and any prime $p$. 
In particular, the equality $[L_n:K_n]= p^{2nr}$ holds. 
\end{Thm}

\begin{proof}By Lemma \ref{any-case},  ${\rm Gal}(L_1/L_1\cap K_\infty)
={\rm Gal}(L_1/K_1)$. Then we have the following commutative diagram
$$
\begin{CD}
{\rm Gal}(L_\infty/K_\infty)  @> \Phi_\infty >>  T_p(E)^r  \\
  @V \alpha_1 VV @V \beta_1 VV @.\\
 {\rm Gal}(L_1/K_1)  @>\Phi_1 >> E[p]^r, 
\end{CD}
$$
where $\alpha_1$ is the restriction map and $\beta_1$ is the reduction modulo $p$. Clearly these vertical arrows are surjective. 

Since $\Phi_1$ is an isomorphism by Theorem \ref{160120a}, $\Phi_1\circ \alpha_1$ is surjective. We see that $\Phi_\infty$ is surjective by using 
Nakayama's Lemma. 
This gives rise to the commutative diagram
$$
\begin{CD}
{\rm Gal}(L_\infty/K_\infty)   @> \Phi_\infty \atop \simeq>>  T_p(E)^r  \\
  @V \alpha_n VV @V \beta_n VV @.\\
 {\rm Gal}(L_n/L_n\cap K_\infty)  @>\Phi_n >> E[p^n]^r,
\end{CD}
$$
where $\alpha_n$ is the restriction map and $\beta_n$ is the reduction modulo $p^n$. 
Thus it follows that the restriction of $\Phi_n$ to ${\rm Gal}(L_n/L_n\cap K_\infty)$ is surjective and thus $\Phi_n$ is surjective. 
Since $\Phi_n$ is also injective, $\Phi_n$ is an isomorphism. 
Hence $[L_n:K_n]=p^{2nr}$.  
\end{proof}
\begin{Cor}\label{160120b}
The equality $L_n\cap K_\infty=K_n$ holds for $n\ge 1$. 
\end{Cor}
\begin{proof}
In the proof of Theorem \ref{0813a}, we saw that $\Phi_n$ and its 
restriction to ${\rm Gal}(L_n/L_n\cap K_\infty)$ are isomorphisms 
to $E[p^n]^r$. 
Thus we have ${\rm Gal}(L_n/L_n\cap K_\infty)={\rm Gal}(L_n/K_n)$, 
and the assertion follows. 
\end{proof}
\section{The inertia subgroups of ${\rm Gal}(L_n/K_n)$ on $p$}
\label{inertia}
In this section we estimate the order of the inertia subgroups of 
 ${\rm Gal}(L_n/K_n)$ on $p$. 
We also improve the previous result (cf.\ Theorem 1.1 of \cite{SY}). 
\subsection{The local case}

Let us recall the notation in \S 3 of \cite{SY}. 
Fix a natural number $n$. 
Put ${\cal K}_n:=\Q_p(E[p^n])$ 
and let $\p$ be the prime ideal of ${\cal K}_n$. 
Put ${\cal L}_n:={\cal K}_n([p^n]_E^{-1}A)$ and 
let $\P$ be the prime ideal of ${\cal L}_n$. 

We will investigate the order of the inertia subgroup ${\cal I}_n$ of 
 ${\rm Gal}({\cal L}_n/{\cal K}_n)$. 

We denote the generators of $A$ by $P_1,\ldots,P_r$, 
where $A$ is the fixed free subgroup in $E(\Q)$. 
For each $j$ in $\{1,\ldots,r\}$ we take $T_j$ 
such that $[p^n]_E(T_j)=P_j$. 
The injectivity of the homomorphism 
$$
\Phi^{{\rm loc}}_n:~\mbox{Gal}({\cal L}_n/{\cal K}_n) \rightarrow E[p^n]^{r}~:~
\sigma \mapsto ({}^\sigma T_j- T_j)_j
$$
shows that $[{\cal L}_n:{\cal K}_n]$ divides $p^{2nr}$. 

A key point is to prove the cyclicity of ${\cal I}_n$ and 
we make use of the Tate curves to confirm it. 
Since $E$ has multiplicative reduction at $p$, 
there exists $q$ in $p\Z_p$ such that $E$ is isomorphic over ${\cal M}$ to the Tate curve $E_q$ for some unramified extension ${\cal M}$ over $\Q_p$ of degree at most two (cf.\ \cite{sil}, p.\ 357, Theorem 14.1). 
We note $E_q(\overline{\Q}_p)\simeq \overline{\Q}_p^\ast/q^{\Z}$. 

We write $\varphi$ from $E$ to $E_q$ for the isomorphism over ${\cal M}$.  
We define $p_j$ and $t_j$ in $E_q(\bQ_p)$ by 
$$
\varphi(P_j)=p_j\quad \mbox{and}\quad \varphi(T_j)=t_j \quad (1\le j \le r)
$$
(see \S 1 for $P_j$ and $T_j$). 

Assume that $p\nmid \mbox{ord}_p(q)$. 
We have ${\cal MK}_n={\cal M}(\zeta_{p^n},q^{\frac{1}{p^n}})$. 
We discuss about generators of ${\cal L}_n/{\cal K}_n$. 

We put 
$$
H:=\left\{
\begin{array}{ll}
\Q_p^\ast & \mbox{if}~~{\cal M}=\Q_p\\
\{x\in {\cal M}^\ast ~|~N_{{\cal M}/\Q_p}(x)\in q^\Z\}&
  \mbox{if}~~[{\cal M}:\Q_p]=2. 
\end{array}
\right.
$$
Then we have $q \in H$ and 
$E(\Q_p)\simeq H/q^{\Z}$ via $\varphi$ 
(cf.\ \cite{sil}, p.\ 357, Theorem 14.1). 
We have
$$
E(\Q_p)/[p^n]_{E}(E(\Q_p))\simeq H/
\langle H^{p^n}, q\rangle. 
$$
\subsubsection{}
We consider the case of ${\cal M}=\Q_p$. 
Then $H=\Q_p^\ast$ and 
$$
H=
\left\{
\begin{array}{ll}
\langle p \rangle \times (\Z/p\Z)^\ast \times (1+p\Z_p) 
& \mbox{if $p\ne 2$}\\
\langle 2 \rangle \times (\Z/4\Z)^\ast \times (1+4\Z_p) 
& \mbox{if $p=2$.}
\end{array}
\right.
$$
It follows from $p\nmid \mbox{ord}_p(q)$ that 
\begin{equation}
H/
\langle H^{p^n}, q\rangle=
\left\{
\begin{array}{ll}
\langle 1+p \rangle 
\simeq \Z/p^n\Z
& \mbox{if $p\ne 2$}\\
\langle -1 \rangle \times \langle 5 \rangle \simeq 
\Z/2\Z\times \Z/2^n\Z
& \mbox{if $p=2$.}
\end{array}
\right.
\label{1018a}
\end{equation}
Hence $E(\Q_p)/[p^n]_E(E(\Q_p))$ is 
an abelian group of type $(p^{n})$, $(2^{n},2)$. 
\par
\bigskip
We discuss generators of 
the image of the projection from 
the subgroup 
$\langle p_1,\ldots,p_r,q\rangle/q^{\Z}$ 
to $H/\langle H^{p^n},q\rangle$. 

We first consider the case of $p>2$. 

Since $\Z/p^n\Z$ is a local ring, there is a relation of inclusion between 
every two submodules of $\Z/p^n\Z$. 
By renumbering, we may assume 
$\langle p_j\rangle \subset \langle p_1\rangle $ 
as a subgroup of $H/\langle H^{p^n}, q\rangle$ 
for each $j\leq r$. 

In this case ${\cal L}_n={\cal K}_n(t_1)$ holds. 
\par
\bigskip
Secondly we consider the case of $p=2$. 

Since $\Z/2^n\Z$ is a local ring, there is a relation of inclusion between 
every two submodules of $\Z/2^n\Z$. 
By renumbering, we may assume 
$\langle p_j,-1\rangle \subset \langle p_1,-1\rangle $ 
as a subgroup of $H/\langle H^{2^n}, q\rangle$ 
for each $j\leq r$. 
If the rank of $\langle p_1,\ldots,p_r \rangle$ is two, 
we may assume $p_2 \notin \langle p_1\rangle$. 
Then we have $p_2=-p_1^k$. 
By replacing $p_2$ by $p_2p_1^{-k}$, we may assume $p_2=-1$. 

In this case ${\cal L}_n={\cal K}_n(t_1)$ 
or ${\cal L}_n={\cal K}_n(t_1,\zeta_{2^{n+1}})$ holds.  
\subsubsection{}
We consider the case of $[{\cal M}:\Q_p]=2$. 
Then 
$$
H:=\{x\in {\cal M}^\ast ~|~N_{{\cal M}/\Q_p}(x)\in q^\Z\} 
$$
and we investigate the structure of $H/H^{p^n}$. 

Since $N_{{\cal M}/\Q_p}(q)=q^2$, 
the image of $H$ via $N_{{\cal M}/\Q_p}$ is 
a subgroup in $q^\Z$ of exponent 1 or 2. 
Thus $H$ contains the group $H_0=\langle q \rangle \times U_{{\cal M},1}$ as 
a subgroup of exponent 1 or 2, 
where $U_{{\cal M},1}$ is the subgroup of ${\cal M}^\ast$ with norm 1. 
\par
\bigskip
We first consider the case of $p>2$. 
Since the exponent $[H:H_0]$ is prime to $p^n$, we have 
$$
H/H^{p^n}\simeq 
H_0/H_0^{p^n}\simeq 
\langle q \rangle \times U_{{\cal M},1}/U_{{\cal M},1}^{p^n}.
$$

We investigate generators of $U_{{\cal M},1}/U_{{\cal M},1}^{p^n}$. 
We denote the ring of integers in $\cal M$ by $\cal O$. 
$$
\log~:~1+p{\cal O}\rightarrow p{\cal O}~:~1+x\mapsto \log(1+x)
$$
converges and it gives an isomorphism. 
Since $\log(1+x)$ is a powerseries with coefficients in $\Q_p$, 
it commutes with the action of $\mbox{Gal}({\cal M}/\Q_p)$. 
Specially each element of $1+p{\cal O}$ with norm one 
corresponds to that of $p{\cal O}$ with trace zero. 

We put ${\cal M}:=\Q_p(\sqrt{D})$ for a square-free integer $D$ in $\Z_p$. 
Then we have 
$$
p{\cal O}\cap \ker Tr_{{\cal M}/\Q_p}=p\Z_p\sqrt{D}
$$
and 
$$
(1+p{\cal O})\cap \ker N_{{\cal M}/\Q_p}=\exp(p\Z_p\sqrt{D}). 
$$
Since ${\cal O}^\ast\simeq ({\cal O}/p{\cal O})^\ast\times (1+p{\cal O})$ 
and  
the order of 
$({\cal O}/p{\cal O})^\ast$ 
is prime to $p$, 
$$
U_{{\cal M},1}/U_{{\cal M},1}^{p^n}=\langle \exp(p\sqrt{D})\rangle\simeq 
\Z/p^n\Z. 
$$
We have 
$$
H/\langle H^{p^n},q\rangle\simeq 
\langle \exp(p\sqrt{D})\rangle\simeq 
\Z/p^n\Z. 
$$
By a similar discussion as in the case of ${\cal M}=\Q_p$, 
we may assume $\langle p_j\rangle \subset \langle p_1\rangle $ 
as a subgroup of $H/\langle H^{p^n}, q\rangle$ 
for each $j\leq r$. 

In this case ${\cal ML}_n={\cal MK}_n(t_1)$ holds. 
\par
\bigskip
Secondly we consider the case of $p=2$. 

We have $N_{{\cal M}/\Q_2}{\cal M}^\ast=\langle 2^2 \rangle \times U_{\Q_2}$. 
It follows from the assumption $2\nmid \mbox{ord}_2(\Delta)$ 
that $2\nmid \mbox{ord}_2(q)$. Thus 
there does not exist $y$ in ${\cal M}$ such that 
$N_{{\cal M}/\Q_2}(y)=q$.  
Thus we have 
$N_{{\cal M}/\Q_2}(x)\in q^{\Z}$ if and only if 
$N_{{\cal M}/\Q_2}(x)\in q^{2\Z}$. 

Since $N_{{\cal M}/\Q_2}(q)=q^2$, we have
\begin{equation}
H=q^{\Z}\times U_{{\cal M},1} 
\label{1103a}
\end{equation}
and
$$
H/H^{2^n}\simeq \langle q \rangle \times U_{{\cal M},1}/U_{{\cal M},1}^{2^n}.
$$

We investigate generators of $U_{{\cal M},1}/U_{{\cal M},1}^{2^n}$.
$$
\log~:~1+4{\cal O}\rightarrow 4{\cal O}~:~1+x\mapsto \log(1+x)
$$
converges and it gives an isomorphism. 
We modify discussion in the case of $p>2$. 
Since 
$$
1 \rightarrow 1+4{\cal O}\rightarrow {\cal O}^\ast \rightarrow 
({\cal O}/4{\cal O})^\ast\rightarrow 1 
$$
and
$$
(\O/4\O)^\ast=\langle \sqrt{5} \rangle\times \mu_6\simeq \Z/2\Z\times \Z/6\Z,
$$
we have
$$
1 \rightarrow (1+4{\cal O})\cap U_{{\cal M},1}\rightarrow 
U_{{\cal M},1} \rightarrow 
\mu_6 \rightarrow 1, 
$$
where $\mu_6$ is the group of 6-th roots of unity. 
Thus we have 
$$
U_{{\cal M},1}=\mu_6 \times  \langle \exp(4\sqrt{5})\rangle. 
$$

Let $U_{{\cal M},\pm 1}$ be the subgroup of ${\cal M}^\ast$ with norm 
$\pm 1$. 
Then the norm mapping induces the exact sequence 
$$
1\rightarrow U_{{\cal M},1}\rightarrow U_{{\cal M},\pm 1} \rightarrow 
\langle -1 \rangle \rightarrow 1. 
$$
We note $\varepsilon:=(-1+\sqrt{5})/2$ has norm $-1$. 
Since $\varepsilon^6=1+4(-2\varepsilon+1)$, 
there exists a unit $w$ such that $\varepsilon^6=\exp(4w\sqrt{5})$. 
Since $3$ is an unit, we have $\varepsilon^2=\eta \exp(4w\sqrt{5}/3)$ 
for some $\eta$ in $\mu_3$. 
Thus we have
\begin{equation}
U_{{\cal M},1}=\mu_6 \times  \langle \varepsilon^2 \rangle. 
\label{1103b}
\end{equation}
We also have
\begin{equation}
H/\langle H^{2^n}, q\rangle =
\langle -1 \rangle \times \langle \varepsilon^2
\rangle 
\simeq \Z/2\Z\times \Z/2^n\Z. 
\label{1018b}
\end{equation}
By a similar discussion as in the case of ${\cal M}=\Q_p$, 
we may assume 
$\langle p_j,-1\rangle \subset \langle p_1,-1\rangle $ 
as a subgroup of $\langle H, q\rangle/\langle H^{2^n}, q\rangle$ 
for each $j\leq r$. 
If the rank of $\langle p_1,\ldots,p_r \rangle$ is two, 
we may assume $p_2=-1$. 

In this case ${\cal ML}_n={\cal MK}_n(t_1)$ or 
${\cal ML}_n={\cal MK}_n(t_1,\zeta_{2^{n+1}})$ holds.  
\subsubsection{}

To sum up our discussion, we have the following proposition. 

\begin{Prop}
Assume that $p\nmid {\rm ord}_p(q)$. 
For a suitable change of basis of a maximal free subgroup $A$ of $E(\Q)$, 
the equation 
${\cal ML}_n={\cal MK}_n(\varphi(T_1))$ or 
${\cal ML}_n={\cal MK}_n(\varphi(T_1),\zeta_{2^{n+1}})$ 
holds. The latter case occurs only when $p=2$, 
and then we may assume $\varphi(P_2)=-1$, 
$\varphi(T_2)=\zeta_{2^{n+1}}$. 
\label{cyclic-imp}
\end{Prop}

\subsection{}

Assume $p>2$. We investigate the ramified index 
${\cal ML}_n/{\cal MK}_n$. 
We need the following Lemma. 

\begin{Lem}[\cite{L}, p.118, Theorem 5.1]
Let $G$ be a group and let $M$ be a $G$-module. 
Let $\alpha$ be in the center of $G$. 
Then $\mbox{H}^1(G,M)$ is annihilated by the map 
$x\mapsto \alpha x-x$ on $M$. 
In particular, if this map is an automorphism of $M$, 
then $\mbox{H}^1(G,M)=0$. 
\label{0808a}
\end{Lem}

By the inf rest-rest exact sequence, we have
$$
0\rightarrow H^{1}
(\mbox{Gal}({\cal M}(\zeta_{p^n})/{\cal M}),\mu_{p^n})
\rightarrow H^{1}(G_{{\cal M}},\mu_{p^n})
\rightarrow H^{1}(G_{{\cal M}(\zeta_{p^n})},\mu_{p^n})
^{ {\rm Gal}({\cal M}(\zeta_{p^n})/{\cal M})}.
$$
When $p>2$, $a-1$ is an unit of 
$(\Z/p^n\Z)^\ast\simeq \mbox{Gal}({\cal M}(\zeta_{p^n})/{\cal M})$ 
for a primitive root $a$ of $(\Z/p^n\Z)^\ast$. 
By Lemma \ref{0808a}, we have 
$$
H^{1}(\mbox{Gal}({\cal M}(\zeta_{p^n})/{\cal M}),\mu_{p^n})=0.
$$
Thus we have 
$$
0\rightarrow H^{1}(G_{{\cal M}},\mu_{p^n})
\rightarrow H^{1}(G_{{\cal M}(\zeta_{p^n})},\mu_{p^n})
^{ {\rm Gal}({\cal M}(\zeta_{p^n})/{\cal M})}.
$$
By the Kummer theory we have
$$
{\cal M}^\ast/{\cal M}^{\ast p^n}\hookrightarrow 
({\cal M}(\zeta_{p^n})^\ast/{\cal M}(\zeta_{p^n})^{\ast p^n})
^{{\rm Gal}({\cal M}(\zeta_{p^n})/{\cal M})}
\hookrightarrow {\cal M}(\zeta_{p^n})^\ast/{\cal M}(\zeta_{p^n})^{\ast p^n}. 
$$
Thus we see that the Galois group of 
${\cal MK}_n(u^{\frac{1}{p^n}})/{\cal M}(\zeta_{p^n})$ is 
of type $(p^n,p^n)$, 
where $u=1+p, ~\exp(p\sqrt{D})$. 
We have $[{\cal MK}_n(u^{\frac{1}{p^n}}):{\cal MK}_n]=p^n$. 

We will see that 
${\cal MK}_n(u^{\frac{1}{p^n}})/{\cal M}(\zeta_p)$ is 
a totally ramified extension. 

Suppose that ${\cal MK}_n(u^{\frac{1}{p^n}})/{\cal M}(\zeta_p)$ 
is not a totally ramified extension. 
Since ${\cal MK}_n(u^{\frac{1}{p^n}})/{\cal M}(\zeta_p)$ 
is a Galois extension, 
there exists an intermidiate field ${\cal N}$ 
such that ${\cal N}/{\cal M}(\zeta_p)$ 
is an unramified extension of degree $p$. 

Since ${\cal N}$ is the composite of ${\cal M}(\zeta_p)$ and 
the unramified extension of degree $p$ over ${\cal M}$, 
${\cal N}/{\cal M}$ is an abelian extension of degree $p(p-1)$. 
Since ${\cal M}(\zeta_{p^2})$ is the unique intermidiate field between 
${\cal MK}_n(u^{\frac{1}{p^n}})$ and ${\cal M}(\zeta_p)$ which is 
an abelian extension over ${\cal M}$ of degree $p(p-1)$, there does not exist ${\cal N}$. 
This contradicts the assumption.

Hence ${\cal MK}_n(u^{\frac{1}{p^n}})/{\cal M}(\zeta_p)$ 
is a totally ramified extension. 

When we put 
\begin{equation}
p_1=q^a u^{mp^\nu}w^{p^n}\quad (a\in \Z,~p\nmid m,~\nu \ge 0,~w\in {\cal M}^\ast), 
\label{1101a}
\end{equation}
we have 
$$
t_1=p_1^{\frac{1}{p^n}}=
\zeta_{p^n}^j \times q^{\frac{a}{p^n}}
 u^{\frac{m}{p^{n-\nu}}}w
 \quad 
 (j \in (\Z/p^n\Z)^\ast). 
$$
Since $\zeta_{p^n},~q^{\frac{1}{p^n}}$ are in ${\cal MK}_{n}$, 
we have 
$$
{\cal ML}_n={\cal MK}_{n}(t_1)={\cal MK}_{n}(u^{\frac{1}{p^{n-\nu}}})
$$
and
$$
[{\cal ML}_{n}:{\cal MK}_{n}]=
\left\{
\begin{array}{ll}
p^{n-\nu}&\mbox{if}~~n>\nu\\
1&\mbox{if}~~n\le \nu.
\end{array}
\right.
$$

${\cal MK}_n/{\cal K}_n$ is unramified and 
${\cal ML}_{n}/{\cal MK}_{n}$ is totally ramified. 
It follows from $p>2$ that $[{\cal MK}_n:{\cal K}_n]$ is coprime to 
$[{\cal L}_n:{\cal K}_n]$. 
Thus 
${\cal ML}_n/{\cal L}_n$ is unramified of degree $[{\cal MK}_n:{\cal K}_n]$ 
and  ${\cal L}_n/{\cal K}_{n}$ is totally ramified of degree $[{\cal ML}_{n}:{\cal MK}_{n}]$ holds. 

Let ${\cal I}_n$ be 
the inertia subgroup of ${\rm Gal}({\cal L}_n/{\cal K}_n)$. 
Then we have 
$$
|~{\cal I}_n~|=
\left\{
\begin{array}{ll}
p^{n-\nu}&\mbox{if}~~n>\nu\\
1&\mbox{if}~~n\le \nu .
\end{array}
\right.
$$
\subsection{}
We consider the case of $p=2$. 

We first discuss the case of ${\cal M}=\Q_2$. 
Then ${\cal L}_n$ is contained in ${\cal K}_n(5^{\frac{1}{2^n}},\zeta_{2^{n+1}})$. 
Further ${\cal K}_n(\sqrt{5})/{\cal K}_n$ is unramified. 
$\zeta_{2^{n}}$ is in ${\cal K}_n$ 
and $[{\cal K}_n(\zeta_{2^{n+1}}):{\cal K}_n]=2$. 
Therefore the inertia group of 
${\cal K}_n(5^{\frac{1}{2^n}},\zeta_{2^{n+1}})/{\cal K}_n$ 
is of type $(2^{m},2)$ for some $m \le n-1$. 
Thus 
$$
|~{\cal I}_n~|\le 2^{n-1}\times 2=2^n. 
$$

When $n\ge 2$ and ${\cal L}_n={\cal K}_n(t_1)$, we can improve the estimate. 
Since ${\cal L}_n/{\cal K}_n$ is cyclic, ${\cal I}_n$ is also cyclic. 
Thus we have
$$
|~{\cal I}_n~|\leq 2^{n-1}. 
$$

When $n=1$ we can decide $|~{\cal I}_1~|$ directly 
because ${\cal L}_1$ is contained in ${\cal K}_1(\sqrt{5},\zeta_{4})$. 
If ${\cal L}_1$ is contained in ${\cal K}_1(\sqrt{5})$, 
then $|~{\cal I}_1~|=1$. 
Otherwise, $|~{\cal I}_1~|=2$.

Secondly we discuss the case of ${\cal M}=\Q_2(\sqrt{5})$. 
Then ${\cal ML}_n$ is contained in ${\cal MK}_n(\varepsilon^{\frac{1}{2^{n-1}}},\zeta_{2^{n+1}})$. 
Further $\zeta_{2^{n}}$ is in ${\cal K}_n$ 
and $[{\cal K}_n(\zeta_{2^{n+1}}):{\cal K}_n]=2$. 
Therefore the inertia group of 
${\cal K}_n(\varepsilon^{\frac{1}{2^{n-1}}},\zeta_{2^{n+1}})/{\cal K}_n$ 
is of type $(2^{m},2)$ for some $m \le n-1$. 
Thus 
$$
|~{\cal I}_n~|\le 2^{n-1}\times 2=2^n. 
$$

When $n\ge 2$ and ${\cal ML}_n={\cal MK}_n(t_1)$, we can improve the estimate. 
Since ${\cal L}_n/{\cal K}_n$ is cyclic, ${\cal I}_n$ is also cyclic. 
Thus we have
$$
|~{\cal I}_n~|\leq 2^{n-1}. 
$$

When $n=1$ we can decide $|~{\cal I}_1~|$ directly because ${\cal ML}_1$ 
is equal to ${\cal MK}_1(\zeta_4)=\Q_2(\sqrt{5},\sqrt{2},\zeta_4)$. 
We note ${\cal MK}_1={\cal M}(\sqrt{q})={\cal M}(\sqrt{\pm 2})$. 
If ${\cal ML}_1$ contains $\zeta_4$ then $|~{\cal I}_1~|=2$. 
Otherwise, $|~{\cal I}_1~|=1$. 
Specially if $E(\Q)/E(\Q)\cap [2]_E(E(\Q_2))$ is not cyclic, 
then $t_2=-1$ and thus $|~{\cal I}_1~|=2$. 

\subsection{The global case}
Let $v$ be any prime above $p$ in $K_n$ and $I_v$ the inertia subgroup of ${\rm Gal}(L_n/K_n)$ at $v$. 
Put $I_p:=\langle I_v~|~v|p\rangle$. 
In this subsection we take a basis of the maximal free subgroup $A$ 
of $E(\Q)$ satisfying the assertion of Proposition \ref{cyclic-imp}. 

We first consider the case of $p>2$. 

If $|~I_v~|=1$, $L_n/K_n$ is unramified at $v$. 
Since both $L_n/\Q$ and $K_n/\Q$ are Galois extensions, 
$L_n/K_n$ is unramified at any prime above $p$ in $K_n$. 
Therefore $I_p=1$. 

We assume that $|~I_v~|=p^{n-\nu}$. 

By Proposition \ref{cyclic-imp}, we have 
${\cal ML}_n={\cal MK}_n(\varphi(T_1))$. 
Thus $L_n/K_n(T_1)$ is unramified at $v$. 
Since both $L_n/\Q$ and $K_n(T_n)/\Q$ are Galois extensions, 
$L_n/K_n(T_1)$ is unramified at any prime above $p$ in $K_n(T_1)$. 
Thus there exists the injective homomorphism 
from $I_p$ to $\mbox{Gal}(K_n(T_1)/K_n)$. 
Since $I_p$ is generated by elements in $I_v$ and their conjugate, 
the exponent of $I_p$ is equal to that of $I_v$. 
$\mbox{Gal}(K_n(T_1)/K_n)$ is $G_n$-isomorphic to $E[p^n]$ 
and $E[p^{n-\nu}]$ is unique $G_n$-invariant subgroup of $E[p^n]$ 
of exponent $p^{n-\nu}$. 
Since $E[p^{n-\nu}]$ is irreducible with respect to the action of $G_n$, 
we have 
$$
|~I_p~|=p^{2(n-\nu)}. 
$$

Secondly we consider the case of $p=2$. 

Suppose that ${\cal ML}_n={\cal MK}_n(t_1)$ and $n\ge 2$. 
Then, $|~I_v~|$ is at most $2^{n-1}$. 
Similarly as above, 
$L_n/K_n(T_1)$ is unramified at any prime above $2$ in $K_n(T_1)$. 
Since $\mbox{Gal}(K_n(T_1)/K_n)\simeq E[2^n]$, 
the inequality
$$
|~I_2~|\le 2^{2(n-1)}
$$
holds. 

If $n=1$, $|~I_v~|$ is at most $2$. 
Since $\mbox{Gal}(K_1(T_1)/K_1)\simeq E[2]$, 
the inequality
$$
|~I_2~|\le 2^{2}
$$
holds. 

Suppose that ${\cal ML}_n={\cal MK}_n(t_1,t_2)$. 
Then $L_n/K_n(T_1,T_2)$ is 
unramified at any prime above $2$ in $K_n(T_1,T_2)$. 
Since $I_v$ is of type $(2^m,2)$ for some $m \le n-1$ 
and $\varphi(T_2)=\zeta_{2^{n+1}}$, 
there exists $\alpha$ in $E[2]$ such that 
${}^\sigma T_2=T_2\oplus_E \alpha$ for $\sigma$ in $I_v$ 
and
$$
I_v \hookrightarrow E[2^{n-1}]\times E[2]
$$
via the isomorphism 
$$
\mbox{Gal}(K_n(T_1,T_2)/K_n)\simeq E[2^n]\times E[2^n].
$$
Thus 
the inequality
$$
|~I_2~|\le 2^{2(n-1)}\times 2^{2}=2^{2n}
$$
holds. 

If $n=1$, $|~I_v~|$ equals 2. 
Indeed ${\cal ML}_1=\Q_2(\sqrt{5},\sqrt{2},\zeta_4)$ 
and ${\cal MK}_1={\cal M}(\sqrt{\pm 2})$. 
It follows from $\varphi(T_2)=\zeta_{4}$ that $K_1(T_2)/K_1$ is ramified. 
Thus $K_1(T_1,T_2)/K_1(T_2)$ is unramified. 
Since $\mbox{Gal}(K_1(T_2)/K_1)\simeq E[2]$, 
the inequality
$$
|~I_2~|\le 2^{2}
$$
holds. 

Now we have the following theorem. 
%{\color{red} We note that $E$ has split multiplicative reduction at 2 if and only if ${\cal M}=\Q_2$ (cf. \cite{Sil2}, p.\ 442, Theorem 5.3).}
%
%
%
\begin{Thm}\label{est1}
Assume $p>2$ and $p\nmid\mbox{ord}_p(\Delta)$. 
Then 
the equation $|~I_p~|=p^{2(n-\nu)}$ holds for $n>\nu$ 
and $|~I_p~|=1$ holds for $n\le \nu$. 

Assume $p=2$. 
Then the inequality $|~I_p~|\le p^{2(n+r_{2,n}-2)+\delta_2}$ holds for all $n\ge 1$, 
where $r_{2,n}=1,2$ according as $E(\Q)/E(\Q)\cap [2^n]_E(E(\Q_2))$ 
is cyclic or not, 
and
$$
\delta_2=\left\{\begin{array}{ll}
2& \mbox{if $n=1$ and $r_{2,1}=1$}\\
0 &\mbox{otherwise}.
\end{array}
\right.
$$
\end{Thm}

\begin{Rem}
Note that the authors roughly estimated it as 
$|~\Phi_n(I_p)~|\le p^{4n}$ in \S $4$ of {\rm \cite{SY}}. 
\end{Rem}

When $p=2$ and ${\cal I}_n$ is not cyclic, we may assume $\varphi(T_2)=-1$. Thus $\zeta_4$ is in ${\cal L}_1$. 
We note that $\zeta_4\notin{\cal L}_1$ implies $r_{2,n}=1$. 
\section{The inertia subgroups of 
${\rm Gal}(L_n/K_n)$ on $\ell \ne p$}
In this section we estimate the order of the inertia subgroups of 
 ${\rm Gal}(L_n/K_n)$ on $\ell\ne p$. 
 
\subsection{The local case when $\ell$ is multiplicative}

Let $\l$ be a prime ideal in $L_n$ lying above $\ell$. 
Let ${\cal L}_n$ and ${\cal K}_n$ be the completion of $L_n$ and $K_n$ respectively. Since $E$ has multiplicative reduction at $\ell$, 
$E$ is isomorphic to the Tate curve $E_q$ for some $q$ in 
$\ell\Z_{\ell}$. 
We denote by $\varphi$ the isomorphism from $E$ to $E_q$. 
The isomorphism 
$\varphi$ is defined over an unramified extension $\M$ over $\Q_\ell$ 
of degree at most two. 
We have ${\cal MK}_n={\cal M}(\zeta_{p^n},q^{\frac{1}{p^n}})$. 

We define $p_j$ in $E_{q}(\overline{\Q}_\ell)$ by 
$\varphi(P_j)=p_j~(1\le j \le r)$. 
We put 
$$
H:=\left\{
\begin{array}{ll}
\Q_\ell^\ast & \mbox{if}~~{\cal M}=\Q_\ell\\
\{x\in {\cal M}^\ast ~|~N_{{\cal M}/\Q_\ell}(x)\in q^\Z\}&
  \mbox{if}~~[{\cal M}:\Q_\ell]=2. 
\end{array}
\right.
$$
\subsubsection{}
We consider the case where $\M=\Q_\ell$ and $\ell\ne 2$. 

Since 
$$
\Q_\ell^\ast=\langle l \rangle\times (\Z/\ell\Z)^\ast\times (1+\ell\Z_\ell), 
$$
and the $p^n$-th power mapping is invertible by $\ell \ne p$, 
we have 
$$
H/H^{p^n}=\Q_\ell^\ast/(\Q_\ell)^{\ast p^n}=\langle l \rangle \times \langle 
\zeta_{\ell-1}\rangle
\simeq (\Z/p^n\Z)\times (\Z/p^{m} \Z), 
$$
where we put $m:=\min\{\mbox{ord}_p(\ell-1),n\}$. 
We have
\begin{equation}
H/\langle H^{p^n},q\rangle=\langle l \rangle \times \langle 
\zeta_{\ell-1}\rangle\simeq (\Z/p^\nu\Z)\times (\Z/p^m \Z), 
\label{0811b}
\end{equation}
where $\nu:=\min\{\mbox{ord}_p(\mbox{ord}_\ell(q)), n\}$. 

It follows from (\ref{0811b}) that 
$$
{\cal L}_n\subset {\cal K}_n(\zeta_{p^n(\ell-1)},\ell^{\frac{1}{p^n}})
$$
and 
$$
[{\cal K}_n(\zeta_{p^n(\ell-1)},\ell^{\frac{1}{p^n}}):{\cal K}_n(\zeta_{p^n(\ell-1)})]=p^\nu. 
$$
We also have 
$$
\Q_\ell (\zeta_{p^n(\ell-1)},q^{\frac{1}{p^n}})=
\Q_\ell (\zeta_{p^n(\ell-1)},\ell^{\frac{1}{p^{n-\nu}}}).
$$ 

Since ${\cal L}_n(\zeta_{p^n(\ell-1)})/
{\cal K}_n(\zeta_{p^n(\ell-1)})$ is cyclic, 
there exists $t_j$ (say $t_1$) such that 
${\cal L}_n(\zeta_{p^n(\ell-1)})={\cal K}_n(\zeta_{p^n(\ell-1)}, t_1)$. 

Since ${\cal K}_n(\zeta_{p^n(\ell-1)})/{\cal K}_n$ is unramified, 
the ramification index ${\cal L}_n/{\cal K}_n$ is equal to that of 
${\cal L}_n(\zeta_{p^n(\ell-1)})/{\cal K}_n(\zeta_{p^n(\ell-1)})$. 

On the one hand, 
$\Q_\ell(\ell^{\frac{1}{p^n}})/\Q_\ell$ is a 
totally ramified extension of degree $p^n$. 
On the other hand, 
$\Q_\ell(\zeta_{p^n(\ell-1)})/\Q_\ell$ is an unramified extension 
by $\ell \nmid p^n(\ell-1)$. 
Thus the ramified index of the extension 
$\Q_\ell(\zeta_{p^n(\ell-1)},\ell^{\frac{1}{p^n}})/\Q_\ell$ is $p^n$. 

We put $\mu:=\min\{n,~\mbox{ord}_p(\mbox{ord}_\ell(p_1))\}$. 
Then we have 
$$
{\cal L}_n(\zeta_{p^n(\ell-1)})=
{\cal K}_n(\zeta_{p^n(\ell-1)},t_1)=
\Q_\ell(\zeta_{p^n(\ell-1)},
\ell^{\frac{1}{p^{n-\nu}}},
\ell^{\frac{1}{p^{n-\mu}}})
.$$ 
Hence we have
$$
|~{\cal I}_n~|=\left\{
\begin{array}{ll}
p^{\nu-\mu}&\mbox{if}~~\mu<\nu\\
1 &\mbox{if}~~\mu\ge \nu. 
\end{array}
\right.
$$
If $\mbox{ord}_p(\mbox{ord}_\ell(q)) \le \mu$, 
we see that $|~{\cal I}_n~|=1$ for all $n\ge 1$. 
If $\mbox{ord}_p(\mbox{ord}_\ell(q))>\mu$, 
we see that $|~{\cal I}_n~|$ does not depend on $n$ 
for all $n\ge \mbox{ord}_p(\mbox{ord}_\ell(q))$.

\subsubsection{}
\label{0627a}
We consider the case where $[{\cal M}:\Q_\ell]=2$ and $\ell\ne 2$. 

Since $N_{{\cal M}/\Q_\ell}(q)=q^2$, either 
$N_{{\cal M}/\Q_\ell}H=q^{\Z}$ or $N_{{\cal M}/\Q_\ell}H=q^{2\Z}$ holds. 
We have 
$$
H=\langle u \rangle \times U_{{\cal M},1}
$$
for some $u$ in $\cal M$. 
We may take $u$ satisfying $N_{{\cal M}/\Q_\ell}(u)=q^t$ for $t=1,2$. 
Since ${\cal M}$ is unramified over $\Q_\ell$, 
we have $N_{{\cal M}/\Q_\ell}{\cal O}^\ast=\Z_\ell^\ast$.
If $\mbox{ord}_\ell(q)$ is even, we have $N_{{\cal M}/\Q_\ell}(u)=q$. 
If $\mbox{ord}_\ell(q)$ is odd, we have $N_{{\cal M}/\Q_\ell}(u)=q^2$. 

In the case of $t=2$, 
$u$ is in $\langle q\rangle \times U_{{\cal M},1}$. 
If either $p>2$ or $t=2$ holds, we have 
$$
H/\langle H^{p^n},q \rangle =U_{\M,1}/U_{\M,1}^{p^n}. 
$$
If $p=2$ and $t=1$, we have 
$$
H/\langle H^{p^n},q\rangle=\langle u \rangle \times U_{\M,1}/U_{\M,1}^{p^n}. 
$$

Since 
$$
{\cal O}^\ast=({\cal O}/\ell{\cal O})^\ast \times (1+\ell{\cal O}), 
$$
we have 
$$
U_{\M,1}=\mu_{\ell+1}\times \langle \exp(\ell \sqrt{D})\rangle.
$$

If either $p>2$ or $t=2$ holds, we have 
$$
H/\langle H^{p^n},q\rangle
=\langle \zeta_{\ell+1} \rangle \simeq \Z/p^\mu \Z, 
$$
where we put $\mu:=\min\{\mbox{ord}_p(\ell+1),n\}$.
Then we have
$$
\M {\cal L}_n\subset \M {\cal K}_n(\zeta_{p^n(\ell+1)}). 
$$
It follows from $\ell \nmid p^n(\ell+1)$ that 
$\M {\cal K}_n(\zeta_{p^n(\ell+1)})/\M {\cal K}_n$ is unramified. 
Thus ${\cal L}_n/ {\cal K}_n$ is unramified. 

Hence we have $|~{\cal I}_n~|=1$. 

If both $p=2$ and $t=1$ holds, we have 
$$
H/\langle H^{p^n},q\rangle=
\langle u \rangle \times 
\langle \zeta_{\ell+1} \rangle \simeq 
\Z/2 \Z \times \Z/2^\mu \Z.
$$
We have
$$
\M {\cal L}_n\subset \M {\cal K}_n(\zeta_{2^n(\ell+1)},u^{\frac{1}{2^n}}). 
$$
It follows from $\ell \nmid 2^n(\ell+1)$ that 
$\M {\cal K}_n(\zeta_{2^n(\ell+1)})/\M {\cal K}_n$ is unramified. 
Thus the ramified index of 
${\cal L}_n/ {\cal K}_n$ is less than or equal to two. 

Hence we have $|~{\cal I}_n~|\le 2$. 
\subsubsection{}
We consider the case of $\ell=2$ and $\M=\Q_2$.
 
On the subgroup $ \langle -1 \rangle\times (1+4\Z_2)$ of 
$$\Q_2^\ast=\langle 2 \rangle\times \langle -1 \rangle \times (1+4\Z_2)$$ 
the $p^n$-the power homomorphism is invertible by $2 \ne p$. 
Thus we have 
$$
H/\langle H^{p^n},q\rangle=\langle 2 \rangle \simeq \Z/p^\nu\Z,  
$$
where we put $\nu:=\min\{n,~\mbox{ord}_p(\mbox{ord}_2(q))\}$. 
We have 
$$
{\cal K}_n=\Q_2(\zeta_{p^n},2^{\frac{1}{p^{n-\nu}}}). 
$$
On the one hand, $\Q_2(2^{\frac{1}{p^n}})/\Q_2$ is a totally ramified extension of degree $p^n$. 
On the other hand, $\Q_2(\zeta_{p^n})/\Q_2$ is unramified by $2\nmid p^n$. 
Thus the ramification index of 
$\Q_2(\zeta_{p^n},2^{\frac{1}{p^n}})/\Q_2$ is $p^n$. 

We put $\mu:=\mbox{ord}_p(\mbox{ord}_2(p_1))$. 
Then we have 
$$
{\cal L}_n={\cal K}_n(p_1^{\frac{1}{p^n}})=
\Q_2(\zeta_{p^n},2^{\frac{1}{p^{n-\nu}}},2^{\frac{1}{p^{n-\mu}}}). 
$$
Hence we have
$$
|~{\cal I}_n~|=\left\{
\begin{array}{ll}
p^{\nu-\mu}&\mbox{if}~~\mu<\nu\\
1 &\mbox{if}~~\mu\ge \nu. 
\end{array}
\right.
$$
\subsubsection{}
We consider the case of $\ell=2$ and $[\M:\Q_2]=2$.
Then $q^{\Z}\times U_{{\cal M},1}$ has index at most two in $H$. 

Since $p\ne 2$ and 
$U_{\M,1}=\mu_6 \times \langle \varepsilon^2 \rangle$ 
by (\ref{1103b}), 
we have 
$$
H/\langle H^{p^n},q\rangle 
=\langle \varepsilon^2  \rangle\simeq \Z/p^n\Z 
$$
for $p\ne 3$ and 
$$
H/
\langle H^{p^n},q\rangle 
=\mu_3 \times \langle
\varepsilon^2 \rangle 
\simeq \Z/3\Z\times \Z/3^n\Z 
$$
for $p=3$. 

When $p\ne 3$, we have 
$$
\M {\cal L}_n\subset \M {\cal K}_n(\varepsilon^{\frac{2}{p^n}}).  
$$
Since $p\ne 2$ and $\varepsilon$ is unit, 
$\M {\cal L}_n/\M {\cal K}_n$ is unramified 
and thus ${\cal L}_n/ {\cal K}_n$ is unramified. 

When $p=3$, we have 
$$
\M {\cal L}_{n}\subset 
\M{\cal K}_n(\varepsilon^{\frac{2}{3^n}},\zeta_{3^{n+1}}) 
$$
Since $\Q_2(\zeta_{3^{n+1}})/\Q_2$ is unramified 
and $\varepsilon$ is unit, 
we see that ${\cal ML}_{n}/ {\cal MK}_{n}$ is unramified. 
Hence ${\cal L}_{n}/ {\cal K}_{n}$ is unramified. 

In these cases we have $|~{\cal I}_n~|=1$. 
\subsubsection{}
For a prime $\ell$ at which $E$ has multiplicative reduction, 
we define 
$$
\nu_\ell:=
\left\{
\begin{array}{ll}
\min\{\mbox{ord}_p(\mbox{ord}_\ell(\Delta)),n\}&
\mbox{if the reduction is split.}\\
1&
\mbox{if $p=2$, the reduction is non-split, }\\
&
\mbox{and $\mbox{ord}_\ell(\Delta)$ is even.}\\
0&\mbox{otherwise}. 
\end{array}
\right.
$$
Then the ramification index of 
${\cal L}_{n}/ {\cal K}_{n}$ is less than or equal to $p^{\nu_\ell}$ 
if $E$ has multiplicative reduction at $\ell\ne p$. 

Put $I_\ell:=\langle~I_{\l} \ |\ \l|\ell~\rangle $ as before. 
Since ${\rm Gal}(L_n/K_n)$ is of $p$-th power order, 
each $I_{\l}$ factors through tame quotient, hence it is a cyclic group. 

If $I_{\l}=1$, then $I_\ell=1$. 
Suppose that $I_{\l}\ne 1$. 
The ramification index $K_n(T_j)/K_n$ at $\l$ takes the maximal value 
at some $j$ (say $j=1$). 
If it also takes maximal values at $k\ne 1$, 
then the ramification index of $K_n(T_1,T_k)/K_n(T_1)$ at $\l$ 
is equal to that of $K_n(T_1,T_k)/K_n(T_k)$. 
Since $I_\l$ is cyclic, both 
$K_n(T_1,T_k)/K_n(T_1)$ and 
$K_n(T_1,T_k)/K_n(T_k)$ are unramified at $\l$. 

If the ramification index of $K_n(T_1)/K_n$ at $\l$ is greater than 
that of $K_n(T_k)/K_n$, 
then $K_n(T_1,T_k)/K_n(T_k)$ is ramified at $\l$. 
Since $I_\l$ is cyclic, 
$K_n(T_1,T_k)/K_n(T_1)$ is unramified at $\l$. 

Thus $L_n/K_n(T_1)$ is unramified at $\l$. 
Since $K_n(T_1)/\Q$ is a Galois extension, 
$L_n/K_n(T_1)$ is unramified at $\l$ are unramified at any prime lying above 
$\ell$. 

Therefore we have an upper bound 
$|~I_\ell~|\le p^{2\nu_\ell }$. 
Now we have proved the following theorem. 
\begin{Thm}\label{semistable}
The inequality $|~I_\ell~|\le p^{2\nu_\ell}$ 
holds for a prime $\ell\ne p$ at which $E$ has multiplicative reduction.
\end{Thm}

\subsection{The local case when $\ell$ is potentially good}
Next we consider the case where 
$E$ has potentially good reduction at $\ell$. 
For such a prime $\ell$ we have the following lemma which is a part of Proposition 4.7 of \cite{Gro} due to Raynaud. 
\begin{Lem}\label{pot} Let $E$ be an elliptic curve over $\Q$ which has potentially good reduction at $\ell$. 
Put $m_0=1$ if $p>2$, $m_0=2$ otherwise. 
Then 
the base change $E/K_{m_0}$ has good reduction at any prime in $K_{m_0}$ above $\ell$. 
\label{120317c}
\end{Lem}
\begin{proof} 
Put $q=p^{m_0}$. 
Let ${\cal K}_{m_0}$ the completion of $K_{m_0}$ at a prime $\l$ above $\ell$. 
Let $\rho_{E,p}$ be the $p$-adic Galois representation 
from $G_\Q$ to ${\rm GL}_2(\Z_p)$ 
associated to the $p$-adic Tate module $T_p(E)$. 
It is easy to see that $\rho_{E,p}(G_{K_{m_0}})=1+qM_2(\Z_p)$ is a torsion-free, pro-$p$ group.   
If the restriction mapping $\rho_{E,p}|_{I_{{\cal K}_{m_0}}}$ is non-trivial, the order of $\rho_{E,p}(I_{{\cal K}_{m_0}})$ becomes infinite. 
Since $E$ has potentially good reduction at $\ell$, 
there exists a finite extension ${\cal K}^\prime/{\cal K}_{m_0}$ such that 
$E/{\cal K}^\prime$ has good reduction. 
Thus $|~\rho_{E,p}(I_{{\cal K}_{m_0}})~|$ is less than or equal to 
$[{\cal K}^\prime:{\cal K}_{m_0}]$. 
This gives a contradiction. 

Hence $\rho_{E,p}|_{I_{{\cal K}_{m_0}}}$ is trivial and $E/{\cal K}_{m_0}$ 
has good reduction. 

%When $p=2$, it is easy to see that any torsion elements in $1+2M_2(\Z_2)$ has order at most 2. 
%Since the action of $I_{{\cal K}_{1}}$ on $T_2(E)$ is tame 
%and $\rho_{E,p}(G_{K_1})$ is a pro-2 group, 
%the image $\rho_{E,p}(I_{{\cal K}_{1}})$ is a (topologically) cyclic group. 
%Therefore 
%$\rho_{E,p}(I_{{\cal K}_{1}})$ is a finite of order at most 2 
%or it is of infinite order. 
%The latter case can not happen as we have seen before. 

%Assume that ${\rm ord}_\ell(\Delta)$ is even. 

%Thus if 
%$\rho_{E,p}(I_{{\cal K}_{1}})$ is non-trivial, 
%we can find a quadratic extension 
%$\widetilde{K}_1/K_1$ so that $\rho_{E,p}(I_{\widetilde{{\cal K}}_{1}})$ is trivial, where $\widetilde{{\cal K}}_{1}$ is the completion of $\widetilde{K}_1$ at a prime lying above $\l$. 
\end{proof}

Assume that $(n,p)\not=(1,2)$. 
Let $I_{\l}$ the inertia subgroup of ${\rm Gal}(L_n/K_n)$ at 
a prime $\l$ of $K_n$ lying above $\ell$ with ${\rm ord}_\ell(N)\ge 2$, 
where $N$ is the conductor of $E$. 
Put $I_\ell:=\langle~I_{\l} \ |\ \l|\ell~\rangle $. 
Let ${\cal K}_{n}$ the completion of $K_n$ at $\l$ and $R$ be the ring of integers of ${\cal K}_{n}$. 

By Lemma \ref{120317c}, $E/{K_n}$ has good reduction at $\l$ and then one can take the N\'eron model 
$\mathcal{E}$ of $E$ over ${\cal K}_{n}$. 
By basic properties of N\'eron models 
(cf.\ p.\ 12, Definition 1 and p.\ 16, Corollary 2 of \cite{blr}), 
we have the reduction map  
$E({\cal K}_{n})=\mathcal{E}(R)\stackrel{{\rm red}}{\lra}
\widetilde{E}_{\l}(\F_{\l})$, 
where 
$\widetilde{E}_{\l}$ is the reduction of $E$ at $\l$. 
Then for any $\sigma$ in $I_{\l}$ and $P$ in $E(\overline{{\cal K}}_{n})$ 
we see that 
${\rm red}({}^\sigma P)={\rm red}(P)$. 
Thus ${\rm red}\circ\Phi_n(I_\ell)=\{0\}$ by the definition of 
the $G_n$-isomorphism $\Phi_n$ from ${\rm Gal}(L_n/K_n)$ to $E[p^n]^r$. 
It follows from 
$$E[p^n]^r\stackrel{{\rm red}\atop \sim}{\lra} \widetilde{E}_{\l}[p^n]^r,$$
that $\Phi_n(I_{\l})=\{0\}$ for any $\l$ dividing $\ell$. 
Hence we have $|~I_\ell~|=1$.  

The remaining case is $(n,p)=(1,2)$. 
Since the ramification at $\l$ is tame, $I_\l$ is cyclic. 
Thus we may assume $L_n/K_n(T_1)$ is 
unramified at any prime lying above $\ell$. 
Since $\mbox{Gal}(K_1(T_1))/K_1)\simeq E[2]$, 
we have $|~I_\ell~| \le 2^2$. 

If $l$ is a potentially good prime, we put 
$\nu_\ell=1$ or $0$ according as $(n,p)=(1,2)$ or not. 
Then $|~I_\ell~| \le 2^{\nu_\ell}$. 
\section{Proof of Theorem \ref{main}}
Let us keep our notation in \S 3 and assumptions in Theorem \ref{main}. 
Let $I$ be the subgroup of ${\rm Gal}(L_n/K_n)$ generated by all $I_\ell$ satisfying $\ell|N$, 
where $N$ is the conductor of $E$.  
Put 
$$
s:=\sum_{\ell\ne p}\nu_\ell 
$$
for simplicity. 

We first assume that $p$ is odd. 
We note that $\mbox{Gal}(L_n/K_n)$ is abelian. 
By applying the results in \S 3 and \S 4,     
we have 
$$|~I~|\le \prod_{\ell |N}|~I_\ell~|=\prod_{{\rm ord}_\ell(N)=1}|~I_\ell~|
 \le  p^{2n+2s}.$$
Thus we have
$$
[L_n \cap K^{{\rm ur}}_n:K_n]=\frac{[L_n:K_n]}{[L_n:L_n^I]}
\geq \frac{p^{2nr}}{p^{2n+2s}}=p^{2n(r-1)-2s}
$$
for any $n\geq 1$. 
Here we use $|I_p|\le p^{2n}$ for simplicity. 

Next we assume that $p=2$. 
The constant $r_{2,n}$ and $\delta_2$ are due to Theorem \ref{est1}. 
Then we have 
$$|~I~|\le 2^{2n+2(r_{2,n}-2)+\delta_2+2s}$$
and 
$$
[L_n \cap K^{{\rm ur}}_n:K_n]=\frac{[L_n:K_n]}{[L_n:L_n^I]}
\geq \frac{2^{2nr}}{2^{2n+\delta_2+2s}}
=2^{2n(r-1)-2(r_{2,n}-2)-\delta_2-2s}
$$
for any $n\geq 1$. 

This completes a proof of Theorem \ref{main}.

We define the integer $\nu \ge 0$ by (\ref{1101a}). 
Then $|~I_p~|=p^{2(n-\nu)}$ holds for $n>\nu$, 
and $|~I_p~|=1$ holds for $n\leq \nu$.
Thus our main theorem improves as follows: 
$$
|I|\le p^{2(n-\nu)+2s},\quad 
[L_n \cap K^{{\rm ur}}_n:K_n]\geq p^{2n(r-1)+2\nu-2s}
$$
for $n>\nu$;
$$
|I|\le p^{2s},\quad 
[L_n \cap K^{{\rm ur}}_n:K_n]\geq p^{2nr-2s}
$$
for $n\le \nu$. 

Next, we give a proof of Corollary \ref{cor}. 
If the conductor of $E$ is equal to a prime $p$, 
we have $p\geq 11$, $\Delta~|~p^5$, and 
$G_n\simeq \mbox{GL}_2(\Z/p^n\Z)$ for $n\geq 1$ 
(cf.\ \cite{SY}). 
Thus the assumptions of Theorem \ref{main} hold in this case. 

%We define the integer $\nu \ge 0$ by (\ref{1101a}). 
Since the conductor is equal to $p$, 
we have $|~I~|=|~I_p~|$ 
%By Theorem \ref{est1} we have $|~I_p~|=p^{2(n-\nu)}$ holds for $n>\nu$, 
%and $|~I_p~|=1$ holds for $n\leq \nu$. 
%Since 
%$$
%[L_n \cap K^{{\rm ur}}_n:K_n]=\frac{[L_n:K_n]}{[L_n:L_n^I]}=
%\frac{p^{2nr}}{|~I~|},
%$$
and $s=0$. Thus
we have
$$
\kappa_n=\left\{ \begin{array}{ll}2n(r-1)+2\nu&(n>\nu)\\
2nr&(n\leq \nu).
\end{array}
\right.
$$
This completes the proof. 
\section{$L_1\cap K_\infty=K_1$ for $p=2$}
Let the notations be the same as in \S 2. 
Put $N_1:=L_1\cap K_\infty$. 
Since $N_1/K_1$ is a $G_1$-extension contained in $L_1/K_1$, 
the Galois group $\mbox{Gal}(N_1/K_1)$ is isomorphic to the direct product of some copies of $E[p]$. By our previous paper \cite{SY} the equation $N_1=K_1$ holds for $p>2$. 

In this section, we prove $N_1=K_1$ in the case of $p=2$. 

Put $H_n:=1+p^nM_2(\Z_p)$ for any $n\ge 1$. 
It is isomorphic to ${\rm Gal}(K_\infty/K_n)$ since  
$G_n\simeq {\rm GL}_2(\Z/p^n\Z)$. 
Contrary to the case of $p>2$, we have the issues that the equality $H^2_1=H_2$ does not hold and $H_1/H_2\simeq 
M_2(\Z/2\Z)$ contains $E[2]$ as an irreducible $G_1$-quotient. 
To obtain $N_1=K_1$ in the case of $p=2$ we need more careful analysis.  

\subsection{Maximal abelian extension of $K_1$ in $K_\infty$}
In this subsection we prove $N_1\subset K_2$. 

Instead of $H^2_1$ we consider the subgroup $\mathcal{H}$ of $H_1$ generated by $H^2_1$. 
It is easy to see that $\mathcal{H}$ is a normal subgroup of $H_1$ (and also of ${\rm GL}_2(\Z_2)$). 
Since $H_1/{\cal H}$ is of exponent two, $H_1/{\cal H}$ is an abelian group. 

By the Legendre formula the inequality
$$
\mu\left(\begin{bmatrix}\frac{1}{2}\\j\end{bmatrix}8^j\right)
=-j-\mu(j!)+3j\ge -j-\frac{j}{2-1}+3j=j
$$
holds for $j\ge 0$. Thus 
$$
(1+8M)^{\frac{1}{2}}=\sum_{j=0}^{\infty}\begin{bmatrix}\frac{1}{2}\\j\end{bmatrix}(8M)^j=1+4M-8M^2+\cdots 
$$
converges in $H_2$ for any matrix $M$ in $M_2(\Z_p)$. 
We have $H_2^2=H_3$ and 
$$H_2 \supset \mathcal{H} \supset H_1^2 \supset H_3.$$ 

Since $\det h^2\equiv 1 \bmod{8}$ holds for any $h$ in $H_1$, 
$\det g\equiv 1 \bmod{8}$ holds for any $g$ in ${\cal H}$, 
By direct computation we can check 
$$
{\cal H}=\{g\in H_2~|~\det g \equiv 1 \bmod{8}\}.
$$
We have $[H_2:{\cal H}]=2$ and $[H_1:{\cal H}]=2^5$. 
We can also check $H_3$ is a normal subgroup of $\mathcal{H}$.  

\begin{Lem}\label{even-case2}$N_1\subset K_2$ holds. 
\end{Lem}
\begin{proof}
Since $\mbox{Gal}(N_1/K_1)$ is of exponent two, 
we have
$$
H_1\supset \mbox{Gal}(K_\infty/N_1)\supset {\cal H}.
$$
It follows from $[H_2:{\cal H}]=2$ that 
$\mbox{Gal}(K_\infty/N_1)\cap H_2$ equals to either $H_2$ or ${\cal H}$. 

Suppose that $\mbox{Gal}(K_\infty/N_1)\cap H_2={\cal H}$, 
Then 
$$[H_2:\mbox{Gal}(K_\infty/N_1)\cap H_2]=
[H_2\mbox{Gal}(K_\infty/N_1):\mbox{Gal}(K_\infty/N_1)]=2$$
holds. 
Since $\mbox{Gal}(N_1/K_1)$ is isomorphic to the direct product of some copies of $E[2]$, $[H_1:\mbox{Gal}(K_\infty/N_1)]=2^2,~2^4$ 
and thus 
$[H_1:\mbox{Gal}(K_\infty/N_1)H_2]=2,~2^3$. 
This contradicts that $E[2]$ is irreducible $G_1$-module. 

Therefore $\mbox{Gal}(K_\infty/N_1)\cap H_2=H_2$. 
Now we have $\mbox{Gal}(K_\infty/N_1)\supset H_2$ and 
$N_1 \subset K_2$. 
\end{proof}
\subsection{}
In this subsection we prove $\mbox{Gal}(K_2/N_1)=V_2^{(1)},~V_4$ 
by using the notations in Lemma \ref{submodules}. 

We study the ${\rm GL}_2(\Z/2\Z)$-module $M_2(\Z/2\Z)$ as below. 
\begin{Lem}\label{submodules}
There are exactly four non-trivial ${\rm GL}_2(\Z/2\Z)$-submodules of $V_4:=M_2(\Z/2\Z)$ and they are given by 
$$V_1=\langle\begin{pmatrix}  
1 & 0 \\
0 & 1
\end{pmatrix}\rangle,\ V^{(1)}_2=\langle\begin{pmatrix}  
0 & 1 \\
1 & 1
\end{pmatrix},\ \begin{pmatrix}  
1 & 1 \\
1 & 0
\end{pmatrix}\rangle 
,\ V^{(2)}_2=\langle \begin{pmatrix}  
1 & 1 \\
0 & 1
\end{pmatrix},\ \begin{pmatrix}  
1 & 0 \\
1 & 1
\end{pmatrix}\rangle,$$
and $V_3=M_2(\Z/2\Z)^{{\rm tr}=0}$. 
The relations  
$V_4=V_2^{(1)}\oplus V_2^{(2)}$ and  
$V_2^{(2)}\subset V_3$, $V_1\subset V_2^{(1)}$ holds. 
Further only isotypic $G_1$-quotient modules of $M_2(\Z/2\Z)$ are $V_4/V_3\simeq \Z/2\Z,\ V^{(1)}_2/V_1\simeq \Z/2\Z$ and $V^{(2)}_2\simeq 
(\Z/2\Z)^{\oplus 2}$.  
\end{Lem}
\begin{proof}
Since ${\rm GL}_2(\Z/2\Z)$ is generated by $\begin{pmatrix}  
0 & 1 \\
1 & 0
\end{pmatrix}$ and $\begin{pmatrix}  
0 & 1 \\
1 & 1
\end{pmatrix}$, we have only to compute the orbit decomposition of $M_2(\Z/2\Z)$ under the actions of these two elements. 
\end{proof}
As in the proof of Lemma 2.2 of \cite{SY}, 
the $G_1$-module ${\rm Gal}(N_1/K_1)$ is isomorphic to a copy of 
the irreducible $G_1$-module $E[2]$. 
By Lemma \ref{submodules} we have $\mbox{Gal}(K_2/N_1)=V_2^{(1)},~V_4$. 
In particular, we have $\mbox{Gal}(N_1/K_1)\simeq \{0\},~E[2]$. 
\subsection{The proof of $N_1=K_1$}
In this subsection we decide the inertia group of 
a prime ideal lying above 2 in $K_2$ over $\Q$ 
and we give a proof of $N_1=K_1$.  

Put ${\cal K}_1=\Q_2(E[2])$ and ${\cal K}_2=\Q_2(E[4])$. 
Since $E$ has multiplicative reduction, 
there exists some $q$ in $2\Z_2$ such that $E$ is isomorphic to 
the Tate curve $E_q$ over the unramified extension $\cal M$ of $\Q_2$ 
for ${\cal M}=\Q_2,~\Q_2(\sqrt{-3})$. 
It follows from 
$$
\Delta=q\prod_{n\ge 1}(1-q^n)^{24}
$$
(cf. \cite{sil},p.356)
that 
$$\Q_2(E_q[2])=\Q_2(\sqrt{q})=\Q_2(\sqrt{\Delta}),\quad 
\Q_2(E_q[4])=\Q_2(\sqrt[4]{q},\zeta_4)=\Q_2(\sqrt[4]{\Delta},\zeta_4).$$ 
Since $\mbox{ord}_2(q)$ is odd, 
$\Q_2(\sqrt{q})/\Q_2$ is a totally ramified extension of degree two. 
$\Q_2(\sqrt[4]{q},\zeta_4)/\Q_2$ 
is a totally ramified extension of degree eight. 

Suppose ${\cal M}=\Q_2(\sqrt{-3})$. 
Put $\varphi$ is an isomorphism from $E$ to $E_q$. 
Then ${}^\sigma \varphi=\varphi\circ [-1]_E$ for 
the generator $\sigma$ of $\mbox{Gal}({\cal M}/\Q_2)$. 
Since $\Q_2(\sqrt[4]{q},\zeta_4)/\Q_2$ is totally ramified 
and ${\cal M}/\Q_2$ is unramified, $\Q_2(\sqrt[4]{q},\zeta_4)\cap{\cal M}=\Q_2$. Thus we can prolong $\sigma$ from $\mbox{Gal}({\cal M}/\Q_2)$ to 
$\mbox{Gal}({\cal M}(E[4])/\Q_2)$ such that 
$\sigma$ is the identity on $\Q_2(\sqrt[4]{q},\zeta_4)$. 
For $P$ in $E[4]$ we have 
$$
\varphi(P)={}^\sigma \varphi(P)=\varphi\circ [-1]_E({}^\sigma P).
$$
Thus we have ${}^\sigma P=[-1]_E(P)$. 
Therefore 
$$
\Q_2(E[2])=\Q_2(E_q[2]),\quad \Q_2(E[4])={\cal M}(E_q[4]).
$$

Now we have the following lemma. 

\begin{Lem}
Assume that $G_n\simeq {\rm GL}_2(\Z/2^n\Z)$ for $n=1,2$. 
Then we have
$$
{\cal K}_1=\Q_2(\sqrt{q}),\quad 
{\cal K}_2={\cal M}(\sqrt[4]{q},\zeta_4).
$$
\end{Lem}

The inertia group in ${\cal K}_2/{\cal K}_1$ is equal to 
$\mbox{Gal}({\cal M}(\sqrt[4]{q},\zeta_4)/{\cal M}(\sqrt{q}))$. 
It is  generated by two elements:~
$$
\sqrt[4]{q}\mapsto \sqrt[4]{q},\quad \zeta_4\mapsto -\zeta_4
$$
and
$$
\sqrt[4]{q}\mapsto -\sqrt[4]{q},\quad \zeta_4\mapsto \zeta_4.
$$
Their matrix representation with respect to $E[4]$ is 
equal to those with respect to $E_q[4]=\langle \sqrt[4]{q},\zeta_4 \rangle$ 
and they are
$$
1+2\begin{bmatrix}0&0\\0&1\end{bmatrix},\quad 
1+2\begin{bmatrix}0&0\\1&0\end{bmatrix}, 
$$
respectively. 
By using 
$$
\left\langle 
\begin{bmatrix}0&0\\0&1\end{bmatrix},~\begin{bmatrix}0&0\\1&0\end{bmatrix}
\right\rangle
\cap V_2^{(1)}=\{0\}, 
$$
we have the following lemma.

\begin{Lem}
Assume that $G_n\simeq {\rm GL}_2(\Z/2^n\Z)$ for $n=1,2$. 
The fixed field of $V_2^{(1)}$ in $K_2/K_1$ is 
a totally ramified extension over $K_1$ of degree four. 
\label{1018c}
\end{Lem}

We put $\Q_2N_1={\cal N}_1$. 
By (\ref{1018a}) and (\ref{1018b}) we have 
$$
{\cal N}_1\subset {\cal L}_1\subset {\cal M}(\sqrt{q},\zeta_4).
$$
Thus the ramification index of ${\cal N}_1/{\cal K}_1$ is at most two. 
By Lemma \ref{1018c} we see that 
$\mbox{Gal}(K_2/N_1)=V_2^{(1)}$ does not occur. 

Now we have $\mbox{Gal}(K_2/N_1)=V_4$ and $N_1=K_1$. 

\begin{Thm}\label{p=2}
The equality $N_1=K_1$ holds for $p=2$. 
\end{Thm}

\section{Examples}
In this section we will give elliptic curves which satisfy the condition in Theorem \ref{main}. 
The computation is done by using Mathematica, version 10, and databases Sage \cite{sage} for elliptic curves over $\Q$ and 
\cite{JR} for local fields.  
\subsection{$p=2$}
Let $E$ be the elliptic curve defined by $y^2+xy+y=x^3-141x+624$. 
This elliptic curve has the conductor $N=2\cdot 71^2=10082$, 
the minimal discriminant $\Delta=2^3\cdot 71^3$, and $j$-invariant $2^{-3}\cdot5^3\cdot 19^3$. 
By the criterion of \cite{DD} one can check that $G_n\simeq {\rm GL}_2(\Z/2^n\Z)$ for any $n\ge 1$ since $4t^3(t+1)+j=0$ 
does not have a rational solution in $t$. 
By {\rm \cite{sage}} we see that $E(\Q)\simeq \Z^2$ and it is generated by 
$P_1=(-6,38)$ and $P_2=(6,-1)$. 
%To prove $\delta_2=0$ we check $N_1:=K_\infty\cap L_1=K_1$.  
%However by  Theorem \ref{Mn}, $N_1=M_1=K_2\cap L_1$. Therefore we have only to check $K_2\cap L_1=K_1$.  
%A direct computation shows that 
%$$L_1=K_1[s,t]/(f(s),g(t))\simeq K_1[T]/(F(T)),\ K_2=(K_1)_{h}$$
%where 
%$$f(s)= s^4 + 24 s^3 + 287 s^2- 8366 s +34098 ,\ g(t)= t^4 - 24 t^3 + 275 t^2- 1622 t+4134  ,$$
%and $(K_1)_{h}$ stands for the splitting field over $K_1$ of 
%$$h(y)=-8 y^{12}-44 y^{11}-272035 y^{10}+\cdots\in \Q[y].$$
%The polynomial $h(y)$ is obtained by deleting $x$ with the equation of $E$ from the essential part of the 
%denominators of $4P,\ P=(x,y)\in E$.  
%The isomorphism to $K_1[T]/(F(T))$ from $L_1$ is 
%given by $T=x+y$ and then we have a polynomial $F(T)=T^{16}+520 T^{14}+\cdots$ over $\Q$ of degree 16. 
%Observe the reductions of $F(T)$ and $h(y)$ modulo $53$. Then $F(T)$ mod 53 is separable and it decomposes into 6 polynomials of degree %2 while 
%$h(x)$ mod 53 is completely decomposable. This means that $L_1\cap K_2=K_1$. Therefore $\delta_2=0$. 
%It follows $\varepsilon(N)=1$ since ${\rm ord}_{73}(\Delta)=3$. 

We apply Theorem {\rm \ref{main}} to $E$ for $p=2$. 
Since $E$ has non-split multiplicative reduction at 2, 
we have $\nu_{71}=1$. 
$r_{2,n}=1, 2$ holds. 
Thus $\kappa_1\ge 2\cdot 1 \cdot (2-1)-2(r_{2,n}-2)-\delta_2-2\cdot 1=0$. 
(It becomes an obvious inequality.)
We also have 
$\kappa_n\ge 2n(2-1)-2(r_{2,n}-2)-2\cdot 1\ge 2n-4$ for $n \ge 2$. 
%Then we have $\kappa_n\ge 2n(2-1)+1-2={\color{red}2n-1}$ for $n\ge 1$. 
Hence the class number $h_{\Q(E[2^n])}$ satisfies 
$$2^{2(n-2)}~|~h_{\Q(E[2^n])}$$
for any $n\ge 2$. In this case we can check $\zeta_4=\sqrt{-1}\in {\cal L}_1$.

\subsection{$p=2$ and $r_{2,n}=1$}
Let $E$ be the elliptic curve defined by 
$$h(x,y):=-(y^2+xy+y)+x^3 + x^2 - 55238 x +4974531=0.$$ 
This elliptic curve has the conductor $N=2\cdot 5^2\cdot 313=15650$, 
the minimal discriminant $\Delta=-2^{19}\cdot 5^6 \cdot 313$, and $j$-invariant 
$-2^{-19}\cdot 313^{-1}\cdot 7^{3}\cdot 103^3\cdot 139^3$. 
Further it has split (resp. non-split) multiplicative reduction at $p=2$ (resp. $313$) and 
potentially good reduction at $5$.  

Similarly one can check that $G_n\simeq {\rm GL}_2(\Z/2^n\Z)$ for any $n\ge 1$. 
By {\rm \cite{sage}} we see that $E(\Q)\simeq \Z^2$ and it is generated by 
$P_1=(\ds\frac{37305}{64}, -\frac{6849551}{512})$ and $P_2=(-75, 2987)$. 

A direct computation shows that $L_1$ is obtained by adding the roots of the following two equations to $K_1$:  
$$\begin{array}{l}
f(x)= 64 x^4 - 149220 x^3+ 6883875 x^2 3+ 5695579750 x    -548615793125,\\
 g(x)=x^4+ 300 x^3 + 110850 x^2  - 56367500 x  +4518668125 . 
\end{array}
$$
These polynomials are obtained as follows. Firstly we compute $$2P=(f_1(x,y),g_1(x,y)),\ f_1,g_1\in \Q(x,y)$$ for $P=(x,y)$. 
For $P_1$, we have the system of algebraic equations 
$$f_1(x,y)=\ds\frac{37305}{64},\ g_1(x,y)=-\frac{6849551}{512},\ h(x,y)=0.$$
By deleting $y$ we obtain $f(x)$ as a unique common factor. 
Similarly we obtain $g(x)$ from $P_2$.   

Since $E$ has split multiplicative reduction at $p=2$, we have ${\cal M}=\Q_2$. 
By using \cite{JR} we see that 
$${\cal K}_1=\Q_2(\sqrt{-2}),\ {\cal L}_1=\Q_2(\sqrt{-2},\sqrt{-3},\sqrt{-10})={\cal K}_1(\sqrt{-3}).$$
Therefore $\zeta_4=\sqrt{-1}\not\in {\cal L}_1$ and hence $r_{2,n}=1$.   

We now apply Theorem {\rm \ref{main}} to $E$ for $p=2$. 
Since $E$ has potentially good reduction at 5, $\nu_{5}=1,0$ 
according as $n=1$ or $n\ge 2$. 
Since $E$ has non-split reduction at 313 and $\mbox{ord}_{313}(\Delta)$ 
is odd, $\nu_{313}=0$. 
Then we have $\kappa_1\ge 2\cdot 1 \cdot (2-1)-2(1-2)-2-2\cdot (1+0)=0$ 
and $\kappa_n\ge 2n \cdot (2-1)-2(1-2)-0-2\cdot (0+0)=2n+2$ for $n\ge 2$. 
Hence the class number $h_{\Q(E[2^n])}$ satisfies 
$$2^{2n+2}~|~h_{\Q(E[2^n])}\ (n\ge 2).$$

\subsection{$p=3$}
Let $E$ be the elliptic curve defined by $y^2+xy=x^3+543x+10026$. 
This elliptic curve has the conductor $N=3\cdot 67^2=13467$, 
the minimal discriminant $\Delta=-3^{11}\cdot 67^3$, and $j$-invariant $3^{-11}\cdot 389^3$. 
By \cite{sage} we see that $G_1\simeq {\rm Gal}(\Z/3\Z)$ and $E(\Q)\simeq \Z^2$ whose generators are given by 
$P_1=(-13,35)$ and $P_2=(39,282)$. Then we can apply 
the criterion of \cite{Elkies} (see also the $j$-invariant in p.\ 961 of \cite{DD}) for $G_2$ to obtain  
 $G_2\simeq {\rm GL}_2(\Z/3^2\Z)$. 
 Therefore the conditions in Theorem \ref{main} for $E$ is fulfilled. 
It follows from $r=2$, $\nu_{67}=0$ that $\kappa_n\ge 2n(2-1)=2n$. 
Hence the class number $h_{\Q(E[3^n])}$ satisfies 
$$3^{2n}~|~h_{\Q(E[3^n])}$$
for each $n\ge 1$.

\par
\bigskip
\par
\bigskip
\begin{flushright}
Fumio SAIRAIJI, \\
Faculty of Nursing,\\
Hiroshima International University,\\
Hiro, Hiroshima\\
737-0112, 
JAPAN\\
e-mail:~sairaiji@it.hirokoku-u.ac.jp
\par
\bigskip
\par
\bigskip
Takuya YAMAUCHI,\\
Mathematical Institute,\\
Tohoku University\\
6-3, Aoba, Aramaki, Aoba-Ku, Sendai 980-8578, 
JAPAN \\
e-mail:~yamauchi@math.tohoku.ac.jp
\end{flushright}


\begin{thebibliography}{99}
{ 
\bibitem{bas}M.\ I.\ Bashmakov, 
The cohomology of abelian varieties over a number field, 
Uspehi Mat.\ Nauk 27 (1972) % no. 6 (168) 
25-66. 
English translation: Russian Math. Surveys 27 (1972) 
%no. 6, 
2-70.
%
%
%
\bibitem{blr} S.\ Bosch, W.\ L\"utkebohmert, and M.\ Raynaud, 
N\'eron models. 
Ergebnisse der Mathematik und ihrer Grenzgebiete (3), 21. 
Springer-Verlag, Berlin, (1990).
%
%
%
\bibitem{DD}T.\ Dokchitser and V.\ Dokchitser, Surjectivity of mod $2^n$ representations of elliptic curves, 
Math. Z. 272 (2012), no. 3-4, 961-964. 
%
%
\bibitem{Elkies}N.\ Elkies, Elliptic curves with 3-adic Galois representation surjective mod 3 but not mod 9, preprint 2006.
%
%
\bibitem{G} R.\ Greenberg, 
Iwasawa theory past and present, 
Class field theory its centenary and prospect (Tokyo, 1998), 
Adv.\ Stud.\ Pure Math.\ 30, Math.\ Soc.\ Japan, Tokyo, 2001, 
335-385.
%
%
\bibitem{FKY}T.\ Fukuda, K.\ Komatsu and S.\ Yamagata, 
Iwasawa $\lambda$-invariants and Mordell-Weil rank of abelian varieties with complex multiplication, Acta Arithmetica 127 (2007) %, 
305-307. 
%
%
\bibitem{Gro}A. Grothendieck,  Modeles de Neron et Monodrome Expose IX, SGA 7
%
%
\bibitem{H} T. Hiranouchi, Class numbers associated to elliptic curves over $\Q$ 
with good reduction at $p$, preprint 2016. 
%
%
\bibitem{JR} John W.\ Jones and David P.\ Roberts, 
Database of Local Fields, available at https://math.la.asu.edu/\~{}jj/localfields/.%
%
\bibitem{L}S.\ Lang, 
Elliptic curves: Diophantine analysis, 
Grundlehren der Mathematischen Wissenschaften, 231. Springer-Verlag, Berlin-New York, 1978.
%
%
\bibitem{LW}T.\ Lawson and C.\ Wuthrich, 
Vanishing of some Galois cohomology groups for elliptic curves, Elliptic Curves, Modular Forms and Iwasawa Theory - Conference in honour of the 70th birthday of John Coates: Elliptic Curves, Modular Forms and Iwasawa Theory pp 373-399.
%
%
\bibitem{MO} J.\ -F.\ Mestre and J.\ Oesterl\'{e}, 
Courbes de Weil semi-stables de discriminant une puissance $m$-i\'{e}me, 
J.\ Reine Angew.\ Math.\  400 (1989) 173-184. 
%
%
\bibitem{sage} Sage, a database  available at http://www.lmfdb.org/EllipticCurve/.
%
%
\bibitem{SY}F.\ Sairaiji and T.\ Yamauchi, On the class numbers of the fields of the 
$p^n$-torsion points of certain elliptic curves over $\Q$, 
J.\ Number Theory 156 (2015) 277-289
%
%
%\bibitem{serre}J.\ -P.\ Serre, 
%Abelian $l$-adic representations and elliptic curves, 
%With the collaboration of Willem Kuyk and John Labute. 
%Revised reprint of the 1968 original. 
%Research Notes in Mathematics 7. A K Peters, Ltd.\ , 
%Wellesley, MA, 1998. 
%
%
\bibitem{Ser}J.\ P.\ Serre, 
Proprietes galoisiennes des points d'ordre fini des courbes elliptiques, 
Invent.\ Math.\ 15 (1972) , no. 4, 259-331. 
%
%
\bibitem{sil}J.\ H.\ Silverman, 
The arithmetic of elliptic curves, second edition, 
Graduate Texts in Mathematics 106, Springer, New York, 1986.
%
%
\bibitem{Sil2}J.\ H.\ Silverman, 
Advanced topics in the arithmetic of elliptic curves, 
Graduate Texts in Mathematics 151,  
Springer, New York, 1994.
}
\end{thebibliography}
\end{document}